\documentclass[11pt,reqno]{article}
\usepackage{amsmath,amsthm,amssymb}

\numberwithin{equation}{section}

\newcommand{\di}{{}_{1}}
\newcommand{\dii}{{}_{2}}

\newcommand{\mo}{{}_{(0)}}

\newcommand{\moi}{{}_{(-1)}}

\newtheorem{theorem}[subsection]{Theorem}
\newtheorem{lemma}[subsection]{Lemma}

\newtheorem{corollary}[subsection]{Corollary}

\theoremstyle{definition}
\newtheorem{definition}[subsection]{Definition}

\newcommand{\bk}{\mathbb{F}}

\newcommand{\deltatilde}{\widetilde{\delta}}

\title{Comodule Hom-coalgebras}

\author{Tao Zhang}

\date{}

\begin{document}

\maketitle
{\bf Abstract.} We introduce the concept of comodule Hom-coalgebras 
and show that comodule Hom-coalgebras can be deformed from comodule coalgebras via endomorphisms.

\section{Introduction and Main Results}
\label{sec:intro}
Hom-type algebras first appeared in the form of Hom-Lie algebras, which satisfy a twisted version of the Jacobi identiy.
Hom-algebras have been introduced for the first time in \cite{ms} to construct Hom-Lie algebras using the commutator
bracket. The universal Hom-associative algebra of a Hom-Lie algebra was studied in \cite{yau1}. Module Hom-algebras and Comodule Hom-algebras
have been studied by D. Yau in \cite{yau31,yau32,yau33}. Some other Hom-type algebras such as $n$-ary Hom-Nambu algebras and $n$-ary Hom-Nambu-Lie algebras have been studied in \cite{amm,yau4}.

In this article, we introduce the concept of comodule Hom-coalgebras, the dual vision of module Hom-algebras, study some of their properties.
We also show that comodule Hom-coalgebras can be deformed from comodule coalgebras via endomorphisms. All of our results are dual to D. Yau's work in \cite{yau31,yau33}. The difference between our's and his is that, since we are dealing with comodules and coalgebras, we will find that the Sweedler notions \cite{sweedler,dnr} are more convenient for us to do the work.

Roughly speaking, an \textbf{$H$-comodule Hom-coalgebra} structure on $C$ consists of the following data:

1. $(H,\mu_H,\Delta_H,\alpha_H)$ is a Hom-bialgebra;

2. $(C,\Delta_C,\alpha_C)$ is a Hom-coassociative coalgebra;

3. $C$ has an $H$-comodule structure $\delta \colon C\to H \otimes C$, such that
\begin{equation}
\label{eq:mha''}
 (\alpha_H^2 \otimes \Delta_C)  \circ\delta= \delta_{CC}  \circ \Delta_C.
\end{equation}
We call \eqref{eq:mha''} the \emph{comodule Hom-coalgebra axiom}.  Here $\delta_{CC} \colon  C\otimes C \to H \otimes C\otimes C$ is the map $\delta_{CC}(c\otimes d)=\sum c\moi d\moi\otimes c\mo\otimes d\mo$ for $\delta(c) = \sum c\moi\otimes c\mo$, $\delta(d) = \sum d\moi\otimes d\mo$.

Using Sweedler's notions, \eqref{eq:mha''} can be written as
\begin{equation}
\label{eq:mha'''}
\sum \alpha_H^2(c\moi)\otimes c\mo\di\otimes c\mo\dii = \sum c\di\moi c\dii\moi\otimes c\di\mo\otimes c\dii\mo
\end{equation}
for $c \in C$.  If $\alpha_H^2 = id_H$ (e.g., if $\alpha_H = id_H)$, then \eqref{eq:mha'''} reduces to the usual comodule coalgebra axiom
\begin{equation}
\label{eq:maaxiom}
\sum c\moi\otimes c\mo\di\otimes c\mo\dii = \sum c\di\moi c\dii\moi\otimes c\di\mo\otimes c\dii\mo.
\end{equation}
In particular, comodule coalgebras are examples of comodule Hom-coalgebras in which $\alpha_H=id_H, \alpha_C= id_C$.

On the other direction, we can construct comodule Hom-coalgebras from comodule coalgebras,
as our main results Theorem \ref{thm:char} and  Theorem \ref{thm:deform} show. The first Theorem gives an alternative characterization of comodule
Hom-coalgebras and the second one shows that we can deform  comodule coalgebras into comodule Hom-coalgebras via endomorphisms.

\begin{theorem}
\label{thm:char}
Let $H = (H,\mu_H,\Delta_H,\alpha_H)$ be a Hom-bialgebra, $C = (C,\Delta_C,\alpha_C)$ be a Hom-coassociative coalgebra, and $\delta \colon C\to H \otimes C$ be an $H$-comodule structure on $C$.  Then the following statements hold.
\begin{enumerate}
\item
The map
\begin{equation}
\label{eq:rhotilde}
\deltatilde =  (\alpha_H^2 \otimes id_C)\circ\delta  \colon H \otimes C \to C
\end{equation}
gives $C$ another $H$-comodule structure.
\item
The map
\begin{equation}
\label{eq:rho2}
\delta_{MN} = (\mu_H \otimes id_C^{\otimes 2}) \circ (id_H \otimes \tau_{H,C} \otimes id_C) \circ \delta^{\otimes 2} \colon C^{\otimes 2} \to H \otimes C^{\otimes 2}
\end{equation}
gives $C^{\otimes 2}$ an $H$-comodule structure.
\item
The map $\delta$ gives $C$ the structure of an $H$-comodule Hom-coalgebra if and only if $\Delta_C \colon C\to C^{\otimes 2} $ is a morphism of $H$-comodules, where we equip $C^{\otimes 2}$ and $C$ with the $H$-comodule structures \eqref{eq:rho2} and \eqref{eq:rhotilde}, respectively.
\end{enumerate}
\end{theorem}

\begin{theorem}
\label{thm:deform}
Let $H = (H,\mu_H,\Delta_H)$ be a bialgebra and $C = (C,\Delta_C)$ be an $H$-comodule coalgebra via $\delta \colon C \to H \otimes C$.  Let $\alpha_H \colon H \to H$ be a bialgebra endomorphism and $\alpha_C \colon C \to C$ be an coalgebra endomorphism such that
\begin{equation}
\label{eq:alpharho}
\delta  \circ \alpha_C= (\alpha_H \otimes \alpha_C) \circ \delta.
\end{equation}
Write $H_\alpha$ for the Hom-bialgebra $(H,\mu_{\alpha,H} = \mu_H \circ \alpha_H,\Delta_{\alpha,H} = \alpha_H \circ\Delta_H ,\alpha_H)$ and $C_\alpha$ for the Hom-coassociative coalgebra $(C,\mu_{\alpha,C} = \alpha_C \circ \Delta_C,\alpha_C)$.  Then the map
\begin{equation}
\label{eq:rhoalpha}
\delta_\alpha = \alpha_C \circ \delta \colon C \to H \otimes C
\end{equation}
gives the Hom-coassociative coalgebra $C_\alpha$ the structure of an $H_\alpha$-comodule Hom-coalgebra.
\end{theorem}

Consider now a special case of Theorem \ref{thm:deform} when $\alpha_H = id_H$, we have the following corollary.

\begin{corollary}
\label{cor:deform}
Let $H = (H,\mu_H,\Delta_H)$ be a bialgebra, $C = (C,\Delta_C)$ be an $H$-comodule coalgebra via $\delta \colon C \to H \otimes C$, and $\alpha_C \colon C \to C$ be an coalgebra endomorphism that is also $H$-linear.  Then the map $\delta_\alpha$ \eqref{eq:rhoalpha} gives the Hom-coassociative coalgebra $C_\alpha$ the structure of an $H$-comodule Hom-coalgebra, where $H$ denotes the Hom-bialgebra $(H,\mu_H,\Delta_H,id_H)$.
\end{corollary}

The paper is organized as follows. In Section \ref{sec:intro}, we give an introduction of our work and the statement of our main results. In Section \ref{sec:modalg}, we revisit the relevant definitions and prove the first two parts of Theorem \ref{thm:char}.  In Section \ref{sec:deform}, we prove Theorem \ref{thm:deform} and the third part of Theorem \ref{thm:char}.

\section{Preliminaries and Some Lemmas}
\label{sec:modalg}


In this section, we first recall some basic definitions regarding Hom-modules, Hom-associative algebras, Hom-coassociative coalgebras, and Hom-bialgebras.  The first two parts of Theorem \ref{thm:char} will be proved as Lemmas \ref{lem:rhotilde} and \ref{lem:rho2}.  The last part will be proved in the next section.

Throughout the rest of this paper, vector spaces and linear maps are over a field $\bk$ of any characteristic. Given two vector spaces $U$ and $V$, denote by $\tau = \tau_{U,V} \colon U \otimes V \to V \otimes U$ the twist map, i.e., $\tau(u \otimes v) = v \otimes u$.  For a coalgebra $C$ with comultiplication $\Delta \colon C \to C \otimes C$, we use Sweedler's notation for comultiplication: $\Delta(c) = \sum c\di \otimes c\dii $. For a comodule $\delta\colon M\to  C \otimes M$ over a coalgebra $C$, we will write $\delta(m)=\sum m\moi\otimes m\mo$ \cite{dnr}.

A \textbf{Hom-module} is a pair $(V, \alpha)$ \cite{yau1} in which $V$ is a vector space and $\alpha \colon V \to V$ is a linear map.  A morphism $(V, \alpha_V) \to (W, \alpha_W)$ of Hom-modules is a linear map $f \colon V \to W$ such that $\alpha_W \circ f = f \circ \alpha_V$.  We will often abbreviate a Hom-module $(V,\alpha)$ to $V$.  The tensor product of the Hom-modules $(V, \alpha_V)$ and $(W, \alpha_W)$ consists of the vector space $V \otimes W$ and the linear self-map $\alpha_V \otimes \alpha_W$.

A \textbf{Hom-associative algebra} \cite{ms,ms3,yau2} is a triple $(A,\mu,\alpha)$ in which $(A,\alpha)$ is a Hom-module and $\mu \colon A \otimes A \to A$ is a bilinear map such that
\begin{enumerate}
\item
$\alpha \circ \mu = \mu \circ \alpha^{\otimes 2}$ (multiplicativity) and
\item
$\mu \circ (\alpha \otimes \mu) = \mu \circ (\mu \otimes \alpha)$ (Hom-associativity).
\end{enumerate}
If we write $\mu(a\otimes b)=ab$, this means that for any $a, b, c\in A$,
\begin{equation}
\label{eq:mult}
\alpha(ab)=\alpha(a)\alpha(b),
\end{equation}
\begin{equation}
\label{eq:homass}
\alpha(a)(bc)=(ab)\alpha(c).
\end{equation}
A \textbf{morphism} $f \colon (A,\mu_A,\alpha_A) \to (B,\mu_B,\alpha_B)$ of two Hom-associative algebras is a morphism $f \colon (A,\alpha_A) \to (B,\alpha_B)$ of the underlying Hom-modules such that $f(ab)= f(a)f(b)$ for all $a, b\in A$.

A \textbf{Hom-coassociative coalgebra} \cite{ms,ms3} is a triple $(C,\Delta,\alpha)$ in which $(C,\alpha)$ is a Hom-comodule and $\Delta \colon C \otimes C \to C$ is a bilinear map such that
\begin{enumerate}
\item
$\Delta \circ \alpha =\alpha^{\otimes 2} \circ  \Delta$ (comultiplicativity) and
\item
$ (\alpha \otimes \Delta)\circ\Delta  =(\Delta \otimes \alpha)\circ\Delta$ (Hom-coassociativity).
\end{enumerate}
In what follows, we will also write $\Delta(c) = \sum c\di \otimes c\dii$. So in Sweedler's notation, the above condition means that
\begin{equation}
\label{eq:comult}
\sum \alpha(c)\di\otimes \alpha(c)\dii=\sum \alpha(c\di)\otimes \alpha(c\dii),
\end{equation}
\begin{equation}
\label{eq:homcoass}
\sum \alpha(c\di)\otimes c\dii\di\otimes c\dii\dii=\sum c\di\di\otimes c\di\dii\otimes\alpha(c\dii).
\end{equation}

Suppose that $(C,\Delta_C,\alpha_C)$ and $(D,\Delta_D,\alpha_D)$ are two Hom-coassociative coalgebras.  Their tensor product $C \otimes D$ is a Hom-coassociative coalgebra, with $\alpha_{C \otimes D} = \alpha_C \otimes \alpha_D$ and
\[
\Delta_{C \otimes D} = (id_C \otimes \tau_{D,C} \otimes id_D) \circ (\Delta_C \otimes \Delta_D).
\]
A \textbf{morphism} $f \colon (C,\Delta_C,\alpha_C) \to (D,\Delta_D,\alpha_D)$ of Hom-coassociative coalgebras is a morphism
$f \colon (C,\alpha_C) \to (D,\alpha_D)$ of the underlying Hom-modules such that
$ \Delta_D(f(c))=(f\otimes f)\circ\Delta_D(c)$, i.e. $\sum f(c)\di\otimes f(c)\dii=\sum f(c\di)\otimes f(c\dii)$.


A \textbf{Hom-bialgebra} is a quadruple $(H,\mu,\Delta,\alpha)$ in which:
\begin{enumerate}
\item
$(H,\mu,\alpha)$ is a Hom-associative algebra.
\item
$(H,\Delta,\alpha)$ is a Hom-coassociative coalgebra.
\item
$\Delta$ is a morphism of Hom-associative algebras.
\end{enumerate}

Note that $\Delta$ being a morphism of Hom-associative algebras means that
\begin{equation}
\label{eq:Deltamu}
\Delta \circ \mu = \mu^{\otimes 2} \circ (id_H \otimes \tau \otimes id_H) \circ \Delta^{\otimes 2},
\end{equation}
that is,
\begin{equation}
\label{eq:Deltamu'}
\Delta(ab)=\sum (ab)\di\otimes (ab)\dii=\sum a\di b\di\otimes a\dii b\dii.
\end{equation}


Let $(C,\Delta_C,\alpha_C)$ be a Hom-coassociative coalgebra and $(M,\alpha_M)$ be a Hom-module.  An \textbf{$C$-comodule} structure on $M$ consists of a morphism $\delta\colon M\to  C \otimes M$ of Hom-modules, called the \textbf{structure map}, such that
\begin{equation}
\label{eq:moduleaxiom}
(\alpha_C \otimes \delta)\circ\delta = (\Delta_C \otimes \alpha_M)\circ\delta.
\end{equation}
We will write $\delta(m)=\sum m\moi\otimes m\mo \in C \otimes M$ for $m \in M$.  In this notation, \eqref{eq:moduleaxiom} can be rewritten as
\begin{equation}
\label{eq:moduleaxiom'}
\sum \alpha_C (m\moi)\otimes m\mo\moi\otimes m\mo\mo = \sum (m\moi)\di\otimes (m\moi)\dii\otimes\alpha_M(m\mo).
\end{equation}
If $M$ and $N$ are $C$-comodules, then a \textbf{morphism} of $C$-comodules $f \colon M \to N$ is a morphism of the underlying Hom-modules such that
\begin{equation}
\label{eq:modmorphism}
 \delta_N\circ f =  (id_A \otimes f)\circ \delta_M ,
\end{equation}
that is, $$\sum (f(m))\moi\otimes (f(m))\mo = \sum m\moi\otimes f(m\mo).$$


\begin{definition}
Let $(H,\mu_H,\Delta_H,\alpha_H)$ be a Hom-bialgebra and $(C,\Delta_C,\alpha_C)$ be a Hom-coassociative coalgebra.  An \textbf{$H$-comodule Hom-coalgebra} structure on $C$ consists of an $H$-comodule structure $\delta \colon C\to H \otimes C$ on $C$ such that
\begin{equation}
\label{eq:mha}
 (\alpha_H^2 \otimes \Delta_C)  \circ\delta= \delta_{CC}  \circ \Delta_C.
\end{equation}
\end{definition}
We call \eqref{eq:mha} the \emph{comodule Hom-coalgebra axiom}.  Here $\delta_{CC} \colon  C\otimes C \to H \otimes C\otimes C$ is the map $\delta_{CC}(c\otimes d)=\sum c\moi d\moi\otimes c\mo\otimes d\mo$ (see also \eqref{eq:rho2'} in Lemma \ref{lem:rho2} with $M=N=C$).

If we write $\delta(c) = \sum c\moi\otimes c\mo$ for $c \in C$, then \eqref{eq:mha} can be written as
\begin{equation}
\label{eq:mha'}
\sum \alpha_H^2(c\moi)\otimes c\mo\di\otimes c\mo\dii = \sum c\di\moi c\dii\moi\otimes c\di\mo\otimes c\dii\mo.
\end{equation}

In \cite{yau2}, D.Yau proved that we can deform an associative algebra into a Hom-associative algebra.
Dually, our first lemma \ref{lem:homcoalg} says that we can deform a coassociative coalgebra  into a Hom-coassociative coalgebra.

\begin{lemma}
\label{lem:homcoalg}
Let  $(C,\Delta)$ be a coassociative coalgebra and $\alpha \colon C \to C$ be an coalgebra endomorphism of the coalgebra $(C,\Delta)$.  Define the map
\begin{equation}
\label{eq:mualpha}
\Delta_\alpha =  \Delta  \circ\alpha\colon C\to  C^{\otimes 2} .
\end{equation}
Then $C_\alpha = (C,\Delta_\alpha,\alpha)$ is a Hom-coassociative coalgebra.
\end{lemma}

\begin{proof}  We prove that $\Delta_\alpha$ is Hom-coassociative.
\[
\begin{split}
(\alpha \otimes \Delta_\alpha)\circ\Delta_\alpha(c)
&= \sum \alpha(\alpha(c)\di)\otimes \alpha(\alpha(c)\dii)\di \otimes \alpha(\alpha(c)\dii)\dii\\
&= \sum \alpha^2(c)\di\otimes \alpha^2(c)\dii\di\otimes\alpha^2(c)\dii\dii\\
&= \sum \alpha^2(c)\di\di\otimes \alpha^2(c)\di\dii\otimes\alpha^2(c)\dii\\
&= \sum \alpha(\alpha(c)\di)\di\otimes \alpha(\alpha(c)\di)\dii \otimes \alpha(\alpha(c)\dii)\\
&= (\Delta_\alpha \otimes \alpha)\circ\Delta_\alpha(c).
\end{split}
\]
For the third equality, we use the fact that $(C,\Delta)$ is a coassociative coalgebra. Comultiplicativity of $\alpha$ with respect to $\Delta_\alpha$ can be checked similarly.
\end{proof}

The following Lemma will be needed in proving the first part of Theorem \ref{thm:char}.

\begin{lemma}
\label{lem:rhotilde}
Let $(C,\Delta_C,\alpha_C)$ be a Hom-coassociative coalgebra and $(M,\alpha_M)$ be a $C$-comodule with structure map $\delta \colon M \to C \otimes M$.  Define the map
\begin{equation}
\label{eq:rhot}
\deltatilde = (\alpha_C^2 \otimes id_M)\circ \delta  \colon M \to C \otimes M.
\end{equation}
i.e. $\deltatilde(m)\triangleq\sum \alpha_C^2(m\moi)\otimes m\mo$. Then $\deltatilde$ is the structure map of another $C$-comodule structure on $M$.
\end{lemma}

\begin{proof}
The fact that $\delta$ is a morphism of Hom-comodules means that
\begin{equation}
\label{eq:rhomorphism}
 \delta \circ \alpha_M =(\alpha_C \otimes \alpha_M)\circ \delta,
\end{equation}
i.e.
\begin{equation}
\label{eq:rhomorphism'}
\sum\alpha_M(m)\moi\otimes \alpha_M(m)\mo=\sum\alpha_C(m\moi)\otimes \alpha_M(m\mo).
\end{equation}
First, we show that $\deltatilde$ is a morphism of Hom-comodules:
\[
\begin{split}
\deltatilde \circ \alpha_M(m)
&= \sum\alpha_C^2(\alpha_M(m)\moi)\otimes \alpha_M(m)\mo\\
&= \sum\alpha_C^2(\alpha_C(m\moi))\otimes \alpha_M(m\mo)\quad \text{by \eqref{eq:rhomorphism'}}\\
&= \sum\alpha_C(\alpha_C^2(m\moi))\otimes \alpha_M(m\mo)\\
&= (\alpha_C \otimes \alpha_M) \circ\deltatilde(m).
\end{split}
\]
Second, we show that $\deltatilde$ satisfies \eqref{eq:moduleaxiom}:
\[
\begin{split}
 (\alpha_C \otimes \deltatilde)  \circ\deltatilde(m)
&= (\alpha_C \otimes \deltatilde)(\sum \alpha_C^2(m\moi)\otimes m\mo)\\
&= \sum \alpha_C^3(m\moi)\otimes \alpha_C^2(m\mo\moi)\otimes m\mo\mo\\
&= \sum \alpha_C^2((m\moi)\di)\otimes \alpha_C^2((m\moi)\dii)\otimes \alpha_M(m\mo)\quad \text{by \eqref{eq:moduleaxiom'}}\\
&= \sum (\alpha_C^2(m\moi)\di)\otimes (\alpha_C^2(m\moi)\dii)\otimes \alpha_M(m\mo) \\
&\qquad \qquad\qquad\qquad\qquad\qquad\text{by comultiplicativity of $\alpha_C$}\\
&= (\Delta_C \otimes \alpha_M)\circ\deltatilde(m).
\end{split}
\]
This completes the proof of the Lemma.
\end{proof}

The following Lemma proves the second part of Theorem \ref{thm:char}.  By an $H$-comodule, we mean a comodule over the Hom-coassociative coalgebra $(H,\mu,\alpha)$.

\begin{lemma}
\label{lem:rho2}
Let $(H,\mu_H,\Delta_H,\alpha_H)$ be a Hom-bialgebra, $(M,\alpha_M)$ and $(N,\alpha_N)$ be an $H$-comodule with structure map $\delta_M \colon M\to H \otimes M$ and $\delta_N \colon N\to H \otimes N$ respectively.  Define the map
\begin{equation}
\label{eq:rho2'}
\delta_{MN}= (\mu_H \otimes id_{M\otimes N}) \circ (id_H \otimes \tau_{H,M} \otimes id_M) \circ (\delta_{M}\otimes\delta_{N})
\colon  M\otimes N \to H \otimes M\otimes N.
\end{equation}
\begin{equation}
\label{eq:rho2'}
\delta_{MN}(m\otimes n)\triangleq \sum m\moi n\moi\otimes m\mo\otimes n\mo.
\end{equation}
Then $\delta_{MN}$ is the structure map of an $H$-comodule structure on $M\otimes N$.
\end{lemma}

\begin{proof}
First, we show that $\delta_{MN}$  is a morphism of Hom-comodules:
\[
\begin{split}
&\delta_{MN}\circ \alpha_{M\otimes N}(m\otimes n)=\delta_{MN}(\alpha_{M}(m)\otimes \alpha_{N}(n))\\
&=\sum (\alpha_{M}(m))\moi(\alpha_{N}(n))\moi\otimes m\mo\otimes n\mo \quad\text{by \eqref{eq:rho2'}}\\
&=\sum (\alpha_{H}(m\moi))(\alpha_{H}(n\moi))\otimes \alpha_{M}(m\mo)\otimes \alpha_{N}(n\mo)
\quad \text{by \eqref{eq:rhomorphism}}\\
&= \sum (\alpha_{H}(m\moi n\moi))\otimes \alpha_{M}(m\mo)\otimes \alpha_{N}(n\mo)
\quad \text{by \eqref{eq:comult}}\\
&=(\alpha_H \otimes \alpha_{M\otimes N})\circ  \delta_{MN}(m\otimes n).
\end{split}
\]
To see that $\delta_{MN}$  satisfies \eqref{eq:moduleaxiom} (with $\delta_{MN}$, $H$, and ${M\otimes N}$ in place of $\delta$, $C$, and $M$, respectively), we compute as follows:
\[
\begin{split}
& (\alpha_H \otimes \delta_{MN}) \circ \delta_{MN}(m\otimes n) \\
&= \sum (\alpha_{H}(m\moi n\moi))\otimes (m\mo\moi)(n\mo\moi)\otimes (m\mo\mo)\otimes (n\mo\mo)\quad\text{by \eqref{eq:rho2'}} \\
&= \sum (\alpha_{H}(m\moi)(\alpha_{H}(n\moi))\otimes (m\mo\moi)(n\mo\moi)\otimes (m\mo\mo)\otimes (n\mo\mo)\quad \text{by \eqref{eq:comult}}\\
&= \sum ((m\moi)\di (n\moi)\di)\otimes ((m\moi)\dii (n\moi)\dii)\otimes \alpha_{M}(m\mo)\otimes \alpha_{N}(n\mo)\quad \text{by \eqref{eq:moduleaxiom}}\\
&= \sum ((m\moi n\moi)\di)\otimes ((m\moi n\moi)\dii)\otimes \alpha_{M}(m\mo)\otimes \alpha_{N}(n\mo) \quad \text{by \eqref{eq:Deltamu'}}\\
&=(\Delta_H \otimes \alpha_{M\otimes N})\circ \delta_{MN}(m\otimes n) .
\end{split}
\]
This completes the proof of the Lemma.
\end{proof}

\section{Proof of the Main Theorems}
\label{sec:deform}

\begin{proof}[Proof of Theorem \ref{thm:char}]
The first two parts of the Theorem were proved in Lemma \ref{lem:rhotilde} and Lemma \ref{lem:rho2}. For the third part, we have equip $C$ and $C\otimes C$ with the $H$-comodule structures $\deltatilde$ \eqref{eq:rhot} and $\delta_{CC}$ \eqref{eq:rho2'}, respectively.  Then $\Delta_C$ is a morphism of $H$-comodules if and only if
\[
\begin{split}
\delta_{CC}\circ \Delta_C(c)
&=(id_H\otimes\Delta_C)\circ\deltatilde(c)\quad \text{by \eqref{eq:modmorphism}}\\
&=\sum c\di\moi c\dii\moi \otimes c\di\mo\otimes c\dii\mo\\
&= \sum \alpha_H^2(c\moi)\otimes c\mo\di \otimes c\mo\dii\quad \text{by \eqref{eq:rhot}}\\
&= (\alpha_H^2 \otimes \Delta_C)\circ\delta.
\end{split}
\]
This is exactly the comodule Hom-coalgebra axiom \eqref{eq:mha}. This completes the proof of Theorem \ref{thm:char}.
\end{proof}


\begin{proof}[Proof of Theorem \ref{thm:deform}]
The Hom-coassociative coalgebra $C_\alpha = (C,\Delta_{\alpha,C} = \Delta_C \circ \alpha_C,\alpha_C)$ was discussed in section \ref{sec:modalg}.
The proof that $\delta_\alpha$ is the structure map of an $H$-comodule structure on $C$ is given in Lemmas \ref{lem:rhotilde} and \ref{lem:rho2}.

 In order to show that $\delta_\alpha =\delta \circ \alpha_C $ gives $C_\alpha$ the structure of an $H_\alpha$-comodule Hom-coalgebra, we only need to check that the comodule Hom-coalgebra axiom \eqref{eq:mha} holds.

In this case, the comodule Hom-coalgebra axiom \eqref{eq:mha} means that
\begin{equation}
\label{eq:mhadeform}
 (\alpha_H^2 \otimes \Delta_{\alpha,C})\circ \delta_\alpha = \delta_\alpha{}_{CC}  \circ\Delta_{\alpha,C}.
\end{equation}
Let us compute the two sides of \eqref{eq:mhadeform} as follows:
\[
\begin{split}
&\delta_\alpha{}_{CC}  \circ\Delta_{\alpha,C}(c)\\
&=\sum \alpha_H(\alpha_C(c)\di \moi  \alpha_C(c)\dii\moi) \otimes\alpha_C(c)\di \mo \otimes  \alpha_C(c)\dii\mo
\quad \text{by\eqref{eq:rho2'}} \\
&=\sum \alpha_H(\alpha_H(\alpha_C(c)\di \moi)\alpha_H(\alpha_C(c)\dii\moi)) \otimes\alpha_C(\alpha_C(c)\di \mo) \otimes \alpha_C( \alpha_C(c)\dii\mo)
\quad \text{by\eqref{eq:rhomorphism'}} \\
&=\sum \alpha^2_H(\alpha_C(c)\di \moi)\alpha^2_H(\alpha_C(c)\dii\moi) \otimes\alpha_C(\alpha_C(c)\di \mo) \otimes \alpha_C( \alpha_C(c)\dii\mo)
\quad \text{by\eqref{eq:mult}} \\
&=\sum \alpha^2_H(\alpha_C(c)\moi)\otimes\alpha_C(\alpha_C(c)\mo\di ) \otimes \alpha_C( \alpha_C(c)\mo\dii)
\quad \text{by\eqref{eq:rho2'}} \\
&=\sum \alpha^2_H(\alpha_C(c)\moi)\otimes\alpha_C(\alpha_C(c\mo)\di ) \otimes \alpha_C( \alpha_C(c\mo)\dii)
\quad \text{by\eqref{eq:comult}} \\
&= (\alpha_H^2 \otimes \Delta_{\alpha,C})\circ \delta_\alpha(c).
\end{split}
\]
This completes the proof of Theorem \ref{thm:deform}.
\end{proof}




\vskip0.1cm

College of Mathematics, Henan Normal University, Xinxiang 453007, China

E-mail address: zhangtao@htu.cn

Department of Mathematics and LMAM, Peking University, Beijing 100871, China

E-mail address: zhangtao@pku.edu.cn

\end{document}